\documentclass[a4paper,10pt]{article}
\usepackage[cp1251]{inputenc}
\usepackage{amssymb, amsmath}\usepackage{amsthm}
\usepackage{hyperref}
\newtheorem{theorem}{Theorem}[section]
\newtheorem{lemma}[theorem]{Lemma}

\theoremstyle{definition}

\newtheorem{example}[theorem]{Example}

\theoremstyle{remark}

\theoremstyle{remark}
\newtheorem{problem}[theorem]{Problem}

\numberwithin{equation}{section}

\newtheorem{atheorem}{Theorem}

\newtheorem{corollary}[theorem]{Corollary}
\newtheorem*{corollary*}{Corollary}

\def\th{\theta}

\def\vfi{\varphi}
\newcommand{\si}{\sigma}
\def\Z{\Bbb Z}
\def\N{\Bbb N}

\newcommand{\D}{\mathbb{D}}
\newcommand{\C}{\mathbb{C}}

\newcommand{\Bl}{\Bigl(}
\newcommand{\Br}{\Bigr)}
\newcommand{\Bm}{\Bigr|}

\newcommand{\rr}[1]{re^{i{#1}}}
\newcommand{\inp}{\int_0^{2\pi}}

\def\Ga{\Gamma}
\def\ga{\gamma}

\def\be{\beta}

\def\al{\alpha}
\def\la{\lambda}

\begin{document}

\title{Growth, zero distribution and factorization  of  analytic  functions of moderate  growth in the unit disc}
\author{I. Chyzhykov, S. Skaskiv}







\begin{center}\large \textsc\large Igor Chyzhykov, Severyn Skaskiv \end{center}

\begin{center} \renewcommand{\baselinestretch}{1.3}\bf {\large Growth, zero distribution and factorization  of  analytic  functions of moderate  growth in the unit disc} \end{center}
\vspace{20pt plus 0.5pt}

\begin{abstract}
We give a survey of results on zero distribution and factorization of analytic functions in the unit disc in classes defined by the growth of $\log|f(\rr\theta)|$ in the uniform and integral  metrics. We restrict ourself by the case of finite order of  growth. For a Blaschke product $B$ we obtain a necessary and sufficient condition for the uniform boundedness  of all $p$-means of $\log|B(\rr\theta)|$, where $p>1$.
\end{abstract}


%


Let  $D(z,t)=\{\zeta\in \C: |\zeta-z|<t\}$, $z\in \C$, $t>0$, and
 $\D=D(0,1)$. Denote by $H(\D)$ the class of analytic functions in $\D$.
For  $f\in H(\D)$ we define  the \textit{maximum
modulus} $M(r,f)=\max\{ |f(z)|: |z|=r\}$, $0\le r<1$. The zero sequence of a function $f\in H(\D)$ will be denoted by $Z_f$.
In the sequel, the symbol $C$  with indices stands for  positive constants which depend on parameters indicated.
 We write $a(r)\sim b(r)$ if $\lim _{r\uparrow 1}  a(r)/b(r)=1
    $, $x^+=\max \{x, 0\}$.
Throughout this paper by $(1-w)^\alpha$, $w\in \D$, $\al\in \Bbb
R$, we mean the branch of the power function such that
$(1-w)^{\al}\Bigr|_{w=0}=1$. $BV[a,b]$ stands for the class of functions of bounded variation on $[a,b]$.

We are primary interested in zero distribution of analytic functions  from classes defined by  growth conditions in the unit disc.
The topic is closely related to the problem of factorization of such classes.

Usually, the orders of  growth of an analytic function $f$ in  $\D$
are defined as \[\rho_M[f]=\limsup_{r\uparrow 1} \frac{\log ^+
\log^+ M(r,f)}{-\log(1- r)}, \;\rho_T[f]=\limsup_{r\uparrow 1}
\frac{ \log^+ T(r,f)}{-\log(1- r)},\]
where $T(r,f)=\frac 1{2\pi} \int_0^{2\pi} \log^+ |f(\rr\theta)|\, d\th$.
It is well known  that
\begin{equation}\label{e:0}
\rho_T[f]\le \rho_M[f]\le \rho_T[f]+1,
\end{equation}
 and
all admissible values of the orders  are possible (\cite{ChUMZ}, \cite{Chy_Grish}, \cite{Lin_Bor}).

The paper is organized in the following way. In Sections 1 and 2  we give a survey of results on zero distribution and factorization in subclasses of $H(\D)$ defined by the growth conditions on $T(r,f)$ and $\log M(r,f)$, respectively.
In Section 3 we consider the concept of $\rho_\infty$-order, that goes back to works of C.\ N.\ Linden. This notion allows us to prove several new results for functions with $\rho_M<1$. Finally, in Section 4 we prove a criterion of uniform boundedness of the integral means of $\log|B(\rr\th)|$, where $B$ is a Blaschke product.

We do not consider  zero distribution and factorization either of functions of infinite order or meromorphic functions.
We address the reader who is interested in factorization of meromorphic functions to \cite{Khab06}.

\section{Classes defined by the growth of  $T(r,f)$}
\subsection{Zero distribution and growth of $T(r,f)$}
 To be more precise we start with canonical products.
Let $Z=(z_n)$ be a sequence of complex numbers in $\D$ without accumulation points in $\D$. We define the \textit{exponent of convergence} of $Z$  by $(\inf \emptyset=+\infty)$
\[\mu[Z]=\inf \Bigl\{ \mu\ge 0: \sum_{z_n\in Z} (1-|z_n|)^{\mu+1}<\infty \bigr\} . \]
It is well known \cite{Dj45,Dj48,Na,Ts} that the Djrbashian-Naftalevich-Tsuji \textit{canonical product}
\begin{equation}\label{e:can_prod}
P(z, Z,q)=\prod_{n=1}^\infty E\Bigl( \frac{1-|z_n|^2}{1-\bar z_n z},
q\Bigr),
\end{equation}
where $E(w,0)=1-w$,  \[E(w,q)=(1-w)\exp \{w+
w^2/2 +\dots+ w^q/q\},\quad q\in \N,\] is an analytic function with the zero sequence $Z$ provided that $\sum\limits_{z_n\in Z} (1-|z_n|)^{q+1}<\infty$.
We note that if $q=0$ then $P(z, Z, 0)=C B(z,Z)$, where $C=\prod _{z_n\in Z} {|z_n|}$, $$B(z,Z)=\prod_{z_n\in Z} \frac{\bar z_n(z_n-z)}{|z_n|(1-\bar z_n z)}$$ is the Blaschke product constructed by the sequence $Z$.

Let $n(r,Z_P)$ be the number of zeros in $\overline{D}(0,r)$,
\begin{equation}\label{e:n_ord}
    \rho_n[P]=\limsup_{r\uparrow1} \frac{\log ^+ n(r, P)}{-\log (1-r)},
\end{equation}
be  the order of the counting function of $Z_P$. Under the technical assumption that $0\not \in Z_f$  we also consider the \textit{Nevanlinna counting function}  $N(r,Z_f)=\int_0^r \frac {n(t,Z_f)}t \, dt$. Note that $N(r,Z_f)\le T(r,f)+O(1)$ due to the first fundamental theorem of R.Nevanlinna \cite{Ha}.

 In 1953 Naftalevich \cite{Na}, and in 1956 Tsuji \cite{Ts} proved that
\begin{equation}
    \rho_T[P]=(\rho_n[P]-1)^+  \label{e:rot}.
\end{equation}
 Moreover, $(\rho_n[P]-1)^+$ is equal to the convergence exponent $\mu(Z_P)$, and the order of $N(r,Z_P)$.

%

This result was improved by F. Shamoyan in \cite{Sha78,Sha83}.

Let $\omega\in C^1[0,1)$ be positive, monotone and such that
\begin{equation}\label{e:omega_cond}
    \int_0^1 \omega(t) \, dt<+\infty, \quad \sup_{r\in [r_0, 1)} \Bm \frac {(1-r)\omega'(r)}{\omega(r)}\Bm<q_\omega<+\infty,
\end{equation}
where $r_0\in (0,1)$. If $\omega$ is an increasing function we assume in addition that $0<q_\omega<1$.
The class $A_\omega^*$ consists of analytic functions $f$ in the unit disc satisfying \begin{equation} \label{e:A_omega_class}
    \int_0^1 \omega (r)  T(r,f) \, dr<+\infty.
\end{equation}
If $\omega(r)=(1-r)^{\alpha-1}$, $\alpha>0$, the $A_\omega^*$ coincides with
 Djrbashian's class $A_\alpha^*$ which consists of analytic functions $f$ in the unit disc satisfying
\begin{equation} \label{e:Astar_class}
    \int_0^1 (1-r)^{\alpha-1} T(r,f) \, dr<+\infty.
\end{equation}
Remark that \eqref{e:Astar_class} implies $\rho_T[f]\le \alpha$. On the other hand
$f\in A_{\alpha}^* $ provided that $\alpha> \rho_T[f]$.

\begin{atheorem}[{\cite[Theorem 1]{Sha83}}] Let $\omega$ be a monotone positive function satisfying \eqref{e:omega_cond}, $Z=(z_k)\subset \D$. In order that $Z$ be a sequence of zeros of a function $f\in A_\omega^*$, $f \not\equiv 0$ it is necessary  and sufficient that
\begin{equation} \label{e:sham_cond}
\sum_{z_k\in Z} (1-|z_k|)^2 \omega (|z_k|)<+\infty.
\end{equation}
Moreover, under condition \eqref{e:sham_cond}  Djrbashian's canonical product $P(z,Z,\alpha)$ (see  \eqref{e:djrbash_prod} below) is convergent in $\D$ and belongs to $A_\omega^*$ for $\alpha>q_\omega$.
\end{atheorem}

\subsection{Factorization of classes defined by the growth of $T(r,f)$}
Canonical and parametric
 representations of functions analytic in $\D$ and of finite order of the growth were obtained  \cite{Dj45, Dj48, Dj}
in 1940's    by~M.~M.~Djrbashian using the Riemann-Liouville
fractional int\-eg\-ral.


\begin{atheorem} \label{t:a_alfa_star_faktor}
If $f\in A_\alpha^*$, $\alpha>0$ then $f$ admits a representation
\begin{equation} \label{e:dj_star_repres}
f(z)=C_\lambda z^\lambda P(z, Z_f,\alpha) \exp\Bigl\{\frac \alpha\pi  \int_{\D} \frac{\log |f(\zeta)| (1-|\zeta|^2) \, dm_2(\zeta)}{(1-\bar \zeta z)^{\alpha+2}} \Bigr\},
\end{equation}
where $C_\lambda$ is a complex constant, $\lambda\in \mathbb{Z}_+$, $m_2$ is the planar Lebesgue measure, $P(z, Z_f,\alpha)$ is a canonical product with the zeros $Z_f$, and of the form
\begin{equation} \label{e:djrbash_prod}
    P(z, Z_f, \alpha)=\prod_k \Bl 1- \frac z{z_k}\Br \exp\{ -U_\alpha (z,z_k)\},
\end{equation}
where
\begin{equation*}
    U_\alpha (z,z_k)=\frac {2\alpha}\pi \int_{\D}  \frac{\log |1 -\frac w\zeta| (1-|w|^2) \,dm_2(w)}{(1-\bar w z)^{\alpha+2}}, \quad z\in \D.
\end{equation*}
Moreover, $P(z, Z_f, \alpha)$ converges in $\D$ if and only if
\begin{equation*}
\sum_{z_k\in Z_f} (1-|z_k|)^{\al+1} <+\infty.
\end{equation*}
\end{atheorem}
M.M. Djrbashian \cite{Dj48} noted that 
  $P(z,Z_f,\alpha)$ has the form \eqref{e:can_prod} if $\alpha\in \N$.

Besides the class $A_\alpha^*$, which can be defined by the condition  $$\sup_{0<r<1} \int _0^{2\pi} \biggl(\int_0^r (r-t)^{\al-1} \log^+
|f(te^{i\vfi})|\, dt\biggr)d\vfi<+\infty,$$ M.M.Djrbashian considered the class $ A_\al$ defined by  $$\sup_{0<r<
1}\int _0^{2\pi} \biggl(\int_0^r (r-t)^{\al-1} \log
|f(te^{i\vfi})|\, dt\biggr)^+d\vfi<+\infty.$$  Obviously, $A_\ga^*\subset  A_\al^*\subset A_\al\subset A_\be$, $\ga<\al<\be$. Moreover, the function $g_\alpha(z)=\exp\Bigl\{ \frac 1{(1-z)^{\alpha+1}}\Bigr\}$ belongs to $A_\alpha\setminus A_\alpha^*$.

\begin{atheorem} 
\label{t:dj1} \sl The class $A_\al$, $\alpha>-1$, coincides with the class
of functions represented in the form
\begin{gather}
f(z)=C_\la z^\la B_\al(z)\exp \biggl\{ \int\limits_0^{2\pi}
\frac{d\psi(\th)}{(1-e^{-i\th}z)^{\al+1}} \biggr\}
\equiv C_\la z^\la B_\al(z) \exp\{g_\al(z)\}, \label{e:bar_rep1}
\end{gather}
where $\psi\in BV[0,2\pi]$,  $\sum\limits_{z_k \in Z_f} (1-|z_k|)^{\al+1} <+\infty$; $
B_\al(z)=\prod_k \Bigl(1-\frac z{z_k}  \Bigr)e^{ -W_\al (z,z_k
)}$ is  Djrbashian's product
$$W_\al(z,\zeta)=\int_{|\zeta|}^1 \frac{(1-x)^\alpha}x dx +\sum_k\frac{\Ga(\al+k+1)}{\Ga(\al+1)\Ga(1+k)}
\times $$ $$\times\biggl( (\bar \zeta z)^k\int_{|\zeta|}^1
\frac{(1-x)^\alpha}{x^{k+1}} dx- \Bigl(\frac z\zeta\Bigr)^k
\int_0^{|\zeta|} (1-x)^\al x^{k-1} \, dx\biggr).$$
\end{atheorem}
More general results for arbitrary growth are obtained in \cite{Dj73}.
\section{Classes defined by the growth $\log M(r,f)$}
\subsection{Zero distribution and growth of $\log M(r,f)$}
B.\ Khabibullin \cite{Khab06} considered the following problem.
\begin{problem}
Given a sequence $Z$ in $\D$ without accumulation points in $\D$, find the lowest possible growth of $\log M(r,f)$ in the class of analytic functions $f\not\equiv 0$ vanishing on $Z$.
\end{problem}
An increasing continuous function $d\colon [a,1)\to [0,1)$, where $a\in [0,1)$ is called \cite{Khab06} a  \textit{shift function} if $t<d(t)<1$ for $t\in [a,1)$.

\begin{atheorem}[{\cite[Theorem 1]{Khab06}}] Let $Z$ be a sequence in $\D$, $d$ be convex or concave shift function.
Then there exists a function $f\in H(\D)$, $f\not\equiv0$ such that $Z_f\supset Z$ and $\log M(r,f)\le \frac{C}{d(r)-r} N(r,Z)$ for some positive constant $C$.
\end{atheorem}

Another approach was used by C.N.Linden.  In 1964  (\cite{L_rep}) he established a connection between $\rho_M[P]$ and the zero distribution of $P$, where
$P$ is of the form \eqref{e:can_prod}. To clarify this connection we need some definitions.

Let \[\square(\rr\vfi)=\Bigl\{\rho e^{i\theta}:  r\le \rho \le \frac {1+r}2 , |\theta -\vfi|\le  \frac{\pi(1-r)}2\Bigr\}, \]
and $\nu(\rr\vfi)$ be the number of zeros of $P$ in $\square(\rr\vfi)$.
We define
\begin{equation}\label{e:nu_ord}
\nu_1(r, P)=\max_\varphi \nu (\rr\varphi), \quad
    \nu[P]=\limsup_{r\uparrow1} \frac{\log ^+ \nu_1(r, P)}{-\log (1-r)}, 
\end{equation}

\begin{atheorem}[{\cite[Theorem V]{L_rep}}] \label{t:lin} With the notation above we have
\begin{align}
    \label{e:rom}  \rho_M[P]& \begin{cases} =\nu[P], & \rho_M[P]\ge 1, \\
\le \nu[P]\le 1, & \rho_M[P]< 1. \end{cases}
\end{align}
\end{atheorem}

This result was improved and generalized by F.\ Shamoyan in \cite{Sha78,Sha83}.
We follow the notation of \cite{Sha78}.
Let $\vfi$ be nonnegative increasing function on $(0,+\infty)$. Set
\begin{equation}\label{e:x_phi_class}
    X_\vfi^\infty=\Bigl\{f\in H(\D): \log|f(z)|\le C(f) \vfi\Bl \frac 1{1-|z|}\Br \Bigr\}.
\end{equation}
Assume  that for
$$ \beta_\vfi=\liminf_{x\to+\infty}\frac{x \vfi'(x)}{\vfi(x)}, \quad
 \alpha_\vfi=\limsup_{x\to+\infty}\frac{x \vfi'(x)}{\vfi(x)}$$
 we have $\beta_\vfi \le \alpha_\vfi<+\infty$.

\begin{atheorem}[{\cite[Theorem 1]{Sha78}}] \label{t:sha_78}
Suppose that $\vfi$ satisfies the above conditions.
\begin{itemize}
  \item [i)]  Let $\beta_\vfi>1$. If $f\in X_\vfi^\infty$, $f(0)=1$,  then $\nu(r,Z_f)\le C \vfi\Bigl( \frac 1{1-r}      \Bigr)$ for some positive constant $C$;
  \item [ii)] Let $\beta_\vfi>0$. If $Z$ be an arbitrary sequence in $\D$ such that $\nu(r,Z)\le C \vfi\Bigl( \frac 1{1-r}      \Bigr)$ for some positive constant $C$,   then
      $P(z, Z, \alpha) \in X_\vfi^\infty$ for every $\alpha>\alpha_\vfi+1$.
      \end{itemize}
\end{atheorem}

As we can see, this theorem gives a description of zeros for functions $f\in H(\D)$ of finite order $\rho_M[f]>1$. A counterpart of this result for functions of infinite order is obtained in {\cite[Theorem 2]{Sha83}}.

\subsection{Factorization of classes defined by the growth\\ of $\log M(r,f)$}
In \cite[Theorem I]{L_rep} Linden proved the following result.

\begin{atheorem} \label{t:factorization}
Let $f$ be  analytic in $\D$ and of order $\rho_M[f]\ge 1$. Then
\[ f(z)=z^p P(z) g(z), \]
where $P$ is a canonical product displaying the zeros of $f$, $p$ is nonnegative integer, $g$ is non-zero and both $P$ and $g$ are analytic and of $\rho_M $-order at most $\rho_M[f]$.
\end{atheorem}

Further, in Theorem IV \cite{L_rep}, Linden showed that if $\rho_M[f]<1$ one has $$\max\{ \rho_M[P], \rho_M[g]\}\le \max \{ \rho_M[f], \nu[f]\}.$$


For $\vfi(x)=x^\rho$, $\rho>0$ we denote $X_\rho=X_\vfi^\infty$.

V.\ I.\ Matsaev and Ye.\ Z.\ Mogulski \cite{MaMo} established that if we take $P(z)=P(z, Z_f, s)$, $s\ge [\rho]+1$, $s\in \N$,  in the representation of Theorem \ref{t:factorization}, then the function $g$  has the form
\begin{equation} \label{e:mamo_repres}
    g(z)=\exp  \int_0^{2\pi} S_q(ze^{-i\theta}) \gamma(\theta) \, d\theta, \quad z\in \D ,
\end{equation}
where $q=[\rho]+1$, $S_{q}(z)=\Gamma(q+1)\Bigl(
 \frac 2{(1-z)^{q}} -1\Bigr)$ is the generalized Schwarz kernel, $\gamma$ is a real valued function such that $\gamma\in \mathop{\rm Lip} (q-\rho)$ for noninteger $\rho$, and $\gamma$ satisfies Zygmund's condition $|\gamma(\theta+h)-2\gamma(\theta)+\gamma(\theta-h)|\le Ch$ for integer $\rho$.

In \cite{Sha} F.Shamoyan  showed that non-zero factor $U_\alpha(z)$ in Djrbashian's representation \eqref{e:dj_star_repres}
can be written in the form  \eqref{e:mamo_repres} with $q$ not necessary integer such that  $q>\alpha$, and
 $(k=[q-\alpha])$ $$\int_0^{2\pi} \int_0^{2\pi} \frac{\gamma^{(k)} (t+\theta) -2\gamma^{(k)}(\theta) +\gamma^{(k)}(\theta-t)}{|t|^{1+q-\alpha}} \, dt \, d\theta<+\infty.
$$

In view of relation \eqref{e:0} the following problem arises naturally.

\begin{problem}  { Given $0\le \si \le \rho\le\si +1$,
describe the class $A_\si^\rho$ of analytic functions in $\D$ such that
$\rho_T[f]=\si$,  $\rho_M[f]=\rho$.}
\end{problem}

In \cite{Lin_Bor} Linden constructed canonical products from $A_\si^\rho$ when $\rho>1$, and $\rho-1\le \sigma\le \rho$.
In \cite{Chy_Grish} this problem was solved by the first author under the restriction that $\rho\ge 1$. A solution is given in terms of so called \emph{complete measure} of an analytic function in the sense of Grishin (see \cite{Gr1, FG}).

Let $f\in H(\D)$ be of the form
\begin{gather}
f(z)= C_q z^\la P(z, Z_f,q) \exp\Bigl\{ \int\limits_0^{2\pi}
{S_q(ze^{-i\theta})}{d\psi^*(\theta )}\Bigl\}, \label{e:bar_rep}
\end{gather}
where
$\psi^*\in BV[0,2\pi]$,   $\sum_{z_k\in Z_f} (1-|a_k|)^{q+1} <+\infty$, $\lambda \in \Z_+$, $C_q\in \C$.

Let $M$ be  Borel's subset of
$\overline{\D}$. A \emph{complete measure $\lambda_f$ of genus q
 in the sense of Grishin} is defined by of  as
\begin{equation}\label{e:cm}
  \lambda_f(M)= \sum _{Z_f \cap M} (1-|z_k|)^{q+1}+ \psi(M\cap \partial \D),
\end{equation}
where $\psi$ is the Stieltjes measure associated with $\psi^*$.

A characterization of $\lambda_f$ for $f\in A_\sigma^\rho$ is given in \cite[Theorem 4]{Chy_Grish}.
Another application of $\lambda_f$ can be found in \cite{Chy_illin}
\section{A concept of $\rho_\infty$-order}
Many theorems valid on analytic functions of finite order in $\D$ fail to hold when $\rho_M$-order is smaller than 1 (see e.g.  \cite{Chy_Grish}, \cite{L_rep},  \cite{Li_mp}).

In particular, for a Blaschke product $B$ we always have $0\le \nu[B]\le 1$, so Theorems \ref{t:lin} and \ref{t:sha_78} give no new information on zero distribution of $B$.


The question  arises:

\noindent {\bf  Question.}  {\it What kind of growth characteristic can describe zero distribution in the case when $\rho_M[f]\le 1$?}

For a meromorphic function $f(z)$, $z\in \D$, and $p\ge 1$ we define
\begin{equation*}
    m_p(r,f)=\biggl( \frac 1{2\pi} \int_0^{2\pi} |\log|f(\rr\th)||^p\, d\th\biggr)^{\frac 1p}, \quad 0<r<1.
\end{equation*}

We write
\begin{equation*}
    \rho_p[f]=\limsup_{r\uparrow1} \frac{\log m_p(r,f)}{-\log (1-r)}.
\end{equation*}
A characterization of $\rho_p$-orders can be found in \cite{Li92}.

We define $\rho_\infty$\textit{-order} of $f$ as
\begin{equation*}
    \rho_\infty[f]=\lim_{p\to\infty} \rho_p[f],
\end{equation*}
(existence of the limit follows from the fact that $L_p$-norms are monotone in $p$).
It follows from the First fundamental theorem of Nevanlinna  that $\rho_1[f]=\rho_T[f]$.    Besides, it is known (e.g. \cite{Li_mp}), that $\rho_M [f]\le \rho_p[f]+\frac 1p$ $(p>0)$, which generalizes~\eqref{e:0}. Consequently, $\rho_M[f]\le \rho_\infty[f]$.  Moreover,   Linden \cite{Li_mp} proved that $\rho_\infty[f]=\rho_M[f]$ provided that $\rho_M[f]\ge 1$. Thus, the values $\rho_\infty[f]$ and $\nu[f]$ have similar behavior with respect to the maximum modulus order, when $f$ is a canonical product.

\textit{Remark.} To omit confusion, we have to note that Linden  used the notation $\lambda_\infty(f)$ for $\rho_M[f]$. But he did not consider the  limit $\lim_{p\to\infty} \rho_p[f]$ when $\rho_M[f]<1$.

For a sequence $Z$ in $\D$ with  finite convergence exponent we define $\nu[Z]=\nu[P(z,Z,q)]$ for an appropriate choice of $q$. It is clear that the definition does not depend on $q$.

The following theorem answers the question  posed above.
\begin{atheorem}[{\cite[Theorem 1.1]{CH_pams_11}}] \label{t:rho_inf} Given a sequence  $Z$   in $\D$ such that $\nu=\nu[Z]<\infty$ and an integer $s$ such that $s\ge [\nu]+1$, we  define  the canonical product  $P_s(z)=P(z,Z,s)$.
 Then $\rho_\infty[P_s]=\nu$.
\end{atheorem}
\begin{corollary*}[{\cite[Theorem 1.2]{CH_pams_11}}] \label{c:1} Let $f\in H(\D)$. Then $\nu[f]\le \rho_\infty[f]$.
\end{corollary*}

\begin{example} Let $z_k=1-1/(k \log^2 k)$, $k\in \{3, \dots\}$.
We consider the Blaschke  product $B(z, Z)$. Since $|B|$ is bounded in $\D$, we have $\rho_M[B]=\rho_T[B]=0$, and consequently $\rho_\infty[B]\le 1$.

On the other hand, it is easy to check that
\[n(r, B)\sim \frac 1{(1-r)\log ^2(1-r)}, \quad r\uparrow 1,\] and
\[ \frac {d_1}{(1-r)\log ^2(1-r)}\le  \nu(r)\le \frac {d_2}{(1-r)\log ^2(1-r)}, \quad r\uparrow 1, \] for some positive constants $d_1$, $d_2$.
Hence,  $\nu[B]= 1$, and by the corollary $\rho_\infty[B]=1$.
\end{example}

Taking into account Theorem \ref{c:1} 
we deduce that $\max\{ \rho_M[P], \rho_M[g]\}\le \rho_\infty[f]$ in Theorem \ref{t:factorization}.
A counterpart of Theorem \ref{t:factorization} is valid without restrictions on the value of order.

\begin{atheorem}[{\cite[Theorem 2.1]{CH_pams_11}}]\label{t:factoriz_infty}
Let $f$ be  analytic in $\D$, and  of finite order $\rho_\infty[f]$. Then
\[ f(z)=z^p P(z) g(z), \]
where $P$ is a canonical product displaying the zeros of $f$, $p$ is nonnegative integer, $g$ is non-zero and both $P$ and $g$ are analytic and of $\rho_\infty $-order at most $\rho_\infty[f]$.
\end{atheorem}

Some another applications of the concept of $\rho_\infty$-order such as logarithmic derivative estimates can be found in \cite{CH_pams_11}.

The proof of Theorem  \ref{t:rho_inf} relies on  the inequality $s\ge [\nu]+1$.
 Since the theorem is  not applicable for Blaschke products one may ask what are  relations between zero distribution of a Blaschke product and its $\rho_\infty$-order.

 Here we prove the following Carleson-type result.
 Let $$S(\vfi, \delta)=\{ \rho e^{i\theta}  \in \overline{\mathbb{D}}: \rho \ge 1-\delta, -\pi\delta< \theta -\varphi \le \pi \delta\}$$
be the Carleson square based on the arc $[e^{i(\varphi-\pi\delta)}, e^{i(\vfi+\pi\delta)}]$.

\begin{theorem} \label{t:carl}
Let $Z$ be a sequence in $\D$ such that $\sum_{z_k\in Z} (1-|z_k|)^{s+1}<+\infty$ for some nonnegative integer $s$, $P_s(z)=P(z, Z,s)$.

\begin{itemize}
  \item[i)] Let $\gamma\in (0,s+1]$.
 If
\begin{equation}\label{e:carl_cond}
\sum_{z_n\in S(\vfi, \delta)} (1-|z_n|)^{s+1} \le C_1 \delta^\gamma, \quad \delta \in (0,1),
\end{equation}
for some constant $C_1$ independent of $\vfi$ and $\delta$,
then  for all $p\ge 1$
\begin{equation}\label{e:mp_est_car}
    m_p(r, \log|P_s|)\le \left\{
                                \begin{array}{ll} \displaystyle
                                  C_2 (1-r)^{\gamma-s-1} \Bl \log \frac 1{1-r}+1\Br , & {\gamma\in (0,s+1);} \\
                                  C_2 (\log^2 (1-r)+1), & {\gamma=s+1.}
                                \end{array}
                              \right.
\end{equation}
  \item [ii)]
If $s=0$, and for all $p\ge 1$ we have that $m_p(r, \log|B|)\le K (1-r)^{1-\gamma}$ for some constant $K$ independent of $p$ and $r$ and $\gamma\in (0,1]$ then \eqref{e:carl_cond} holds.
\end{itemize}
\end{theorem}
For a Blaschke product we define
 $\lambda(\vfi, r)=\sum_{z_k\in Z_B\cap S(\vfi, \frac{1-r}2)} (1-|z_k|)$.

\begin{corollary} Let  $B$  be a Blaschke product. Set $$t[B]=\sup \{\gamma \ge 0:  \max_\vfi \lambda(\vfi, r)=O((1-r)^{\gamma}  )\}.$$ Then $\rho_\infty [B]=1-t[B]$.
\end{corollary}

\begin{corollary} If $B$ is an interpolating Blaschke product, then $m_p(r, \log|B|)\le C (\log^2 \frac 1{1-r}+1)$ for all $p\ge 1$.
\end{corollary}




\smallskip

\section{Proof of Theorem \ref{t:carl}}

We start with proving of i).
We write $E_m(\rr\vfi) =S\bigl(\vfi, (1-r)2^{m-1}\bigr)$, $m\in \N$, $E_0(z)=\varnothing$. So $E_1(\rr\vfi)=S(\vfi, 1-r)$, and $E_m(\rr\vfi) =\overline{\D}$ for $m\ge m(r)=\bigl[\log_2 \frac 1{1-r}\bigr]$.

\begin{lemma} \label{l:can_prod_est_com_mes} Let $Z$  be a  sequence in $\D$ such that  $\sum_{z_k\in Z} (1-|z_k|)^{s+1}<\infty$. Suppose that for some $K$ and $\gamma\in (0, s+1]$ condition \eqref{e:carl_cond} holds.
Then
\begin{equation*}
 \sum_{k=1}^\infty \Bm \frac{1-|z_k|^2}{1-z\bar z_k}\Bm^{s+1} \le \begin{cases} \displaystyle \frac {C_3}{(1-|z|)^{s+1-\gamma}}, & \gamma\in (0, s+1),  \\
C_3 \log \frac 1{1-|z|}, & \gamma=s+1.
\end{cases} \quad z\in \D,
\end{equation*}
for some constant $C_3=C_3(s, \gamma)>0$.
\end{lemma}

\begin{proof}[Proof of the lemma] It is easy to see that  $|1-r\rho_k e^{i(\vfi-\theta_k)}|\ge C_4 (1-r)2^m$ for $z_k=\rho_k e^{i\theta _k}\in\D\setminus E_m$ with  some absolute constant $C_4$.
Then
\begin{gather*}
    \sum_{k} \frac{(1-|z_k|^2)^{s+1}}{|1-re^{i\vfi} \bar z_k|^{s+1}} =\sum_{m=1}^{m(r)} \sum_{z_k\in E_m\setminus E_{m-1}} \frac{(1-\rho_k^2)^{s+1}} {|1-r\rho e^{i(\vfi-\theta_k)}|} \le \\ \le  \frac{2^{s+1}}{(C_4(1-r)2^{m-1})^{s+1}} \sum_{z\in E_m} (1-\rho_k)^{s+1}   \le \frac {4^{s+1}}{(C_4(1-r))^{s+1} } \sum_{m=1}^{m(r) } \frac{C_1((1-r)2^{m})^\gamma}{2^{m(s+1)}} \le \\ \le \frac{C_5(s)}{(1-r)^{s+1-\gamma}} \sum_{m=1}^{m(r)} 2^{m(\gamma-s-1)}.
\end{gather*}
The last sum is bounded by a constant  depending on $\gamma$ and $s$ for $\gamma\in (0,s+1)$, and equals $m(r)$ in the case $\gamma=s+1$. This implies the assertion of the lemma.
\end{proof}

We shall need some known results.
\begin{atheorem}[{see \cite[Theorem V.24, p.222; Theorem V.25, p.224]{Tsuji}}] \label{t:tsu_low}
For the canonical product $P_s(z)$
\begin{equation}\label{e:up_est_prod}
    \log^+{|P_s(z)|}\le C_6(s)  \sum_m \Bm \frac {1-|z_m|^2}{1-z\bar z_m}\Bm ^{s+1}, \quad z\in \D, \; \sum_m (1-|z_m|)=+\infty ;
\end{equation}
  if $D_m$ denotes the disc $D\Bl z_m, (1-|z_m|^2)^{s+4}\Br$ then
\begin{equation}\label{e:low_est_prod}
    \log^+\frac1{|P_s(z)|}\le K\log \frac 1{1-|z|} \sum_m \Bm \frac {1-|z_m|^2}{1-z\bar z_m}\Bm ^{s+1}, \quad \frac12\le |z|<1,\; z\not\in \bigcup_m D_m.
\end{equation}
\end{atheorem}

Note that  the following arguments essentially repeat that from \cite[Lemma 1]{Li_mp}.

%
%
%
We first suppose that  $\gamma<s+1$. Then, let   $s\in \N$.
We have to prove that
\begin{equation}\label{e:mp_est}
    \int_0^{2\pi} |\log |P_s(\rr\th)||^p\, d\th\le C_7^p \frac {\log ^p \frac 1{1-r}}{(1-r)^{p(s+1-\gamma)}}.
\end{equation}
We deal with the integral in \eqref{e:mp_est} by covering the range of integration by $[\pi/(1-r)]+1$ intervals of the form $[\tau+r-1, \tau+1-r]$ for $\tau=2k(1-r)$ and $k\in \{ 0, \dots, [\pi/(1-r)]\}$, showing that
\begin{equation}\label{e:mp_est_inter}
    \int\limits_{\tau +r-1}^{\tau+1-r} |\log |P_s(\rr\th)||^p\, d\th\le C_8^p  (1-r)^{-p(s+1-\gamma)+1} \log^p \frac 1{1-r}
\end{equation}
for each $\tau$, where the constant $C_8$ is independent of $\tau$. For convenience and without loss of generality, we may  suppose that $\tau=0$ and  $\frac 34 \le |z_m|<1$.
For given $r$, let $\gamma_r=\{z=\rr \th : r-1\le \th\le 1-r\}$, and $F(r)=\{m:D_m\cap \gamma_r\ne \emptyset\}$, where $D_m$  are the exceptional discs of Theorem \ref{t:tsu_low}. From the definition of the discs $D_m$ and assumptions on $(z_m)$ it follows that
$ 1-4^{-3} \le \frac {1-r}{1-|z_m|}\le 1+4^{-3}$. Hence $\sum_{z_m\in F(r)} (1-|z_m|)^{s+1} \ge \frac {(1-r)^{s+1}}{2^{s+1}} |F(r)|$, where $|F(r)|$ denotes the number of elements in the set $F(r)$.  Thus, by \eqref{e:carl_cond}, we have
\begin{equation}\label{e:E_card}
    |F(r)|\le C_9(1-r)^{\gamma -1-s}.
\end{equation}
 We consider the factorization $P_s=B_1B_2B_3$, where
\begin{equation*}
\begin{split}    B_1(z)&=\prod_{m\not\in F(r)}
E\Bigl( \frac{1-|z_m|^2}{1-\bar z_m z},s\Br, \\
B_2(z)&=\prod_{m\in F(r)} \exp \sum_{j=1}^s \frac 1j \Bl \frac {1-|z_m|^2}{1-z\bar z_m}\Br ^j,  \\
B_3(z)&=\prod_{m\in F(r)} \Bl 1 -\frac {1-|z_m|^2}{1-z\bar z_m}\Br =\prod_{m\in F(r)} \Bl \frac {\bar z_m(z_m -z)}{1-z\bar z_m}\Br.
\end{split}\end{equation*}

First we note that  Theorem \ref{t:tsu_low} and Lemma \ref{l:can_prod_est_com_mes} give
\begin{gather}\nonumber
\int_{r-1}^{1-r} |\log |B_1(\rr\th)||^p d\th  \le \int_{r-1}^{1-r} C_{10}^p \log ^p \frac 1{1-r} \biggl( \sum_m  \Bm \frac {1-|z_m|^2}{1-\rr\th\bar z_m}\Bm^{s+1} \biggr)^p\, d\th  \\
\le C_{10}^p\log ^p \frac 1{1-r} \frac{1}{(1-r)^{p(s+1-\gamma)}} 2(1-r)= \frac {C^p_{11}(s, \gamma) \log^p \frac 1{1-r}}{(1-r)^{p(s+1-\gamma)-1}} .\label{e:B1_est}
    \end{gather}
    Next, the inequality $|1-z\bar z_m|>\frac 12 (1-|z_m|^2)$ yields
    \[
    |\log|B_2(z)||< \sum_{m\in F(r)} \sum_{j=1}^s \frac 1j \Bm \frac {1-|z_m|^2}{1-z\bar z_m}\Bm^{j} \le C_{12} |F(r)|.
    \]
    Hence \eqref{e:E_card} implies
\begin{equation}\label{e:B2_est}
    \int_{r-1}^{1-r} |\log |B_2(\rr\th)||^p d\th\le C_{13}^p  (1-r)^{1-p(s+1-\gamma)}.
\end{equation}
Finally, in \cite[p.124]{Li_mp} it is proved that
\begin{equation}\label{e:B3_est}
    \int_{r-1}^{1-r} |\log |B_3(\rr\th)||^p d\th\le C_{14}|F(r)|^p (1-r)
\end{equation}
Inequality  \eqref{e:mp_est_inter} now follows from \eqref{e:B1_est}--\eqref{e:B3_est}.

In the case  $s=0$  the only difference in the proof is that there is no product $B_2$, and $|B_1(z)|\le (\prod_m |z_m|)^{-1}$.

We now suppose that $\gamma=s+1$. In this case $|F(r)|$ is bounded uniformly in $r$. Instead of \eqref{e:B1_est}, using Lemma
\ref{l:can_prod_est_com_mes}, we obtain
\begin{align}\nonumber
\int_{r-1}^{1-r} |\log |B_1(\rr\th)||^p d\th & \le \int_{r-1}^{1-r} C_{10}^p \log ^p \frac 1{1-r} \biggl( \sum_m  \Bm \frac {1-|z_m|^2}{1-\rr\th\bar z_m}\Bm^{s+1} \biggr)^p\, d\th  \\
&\le 2C_{10}^p\log ^{2p} \frac 1{1-r} (1-r) .\label{e:B1_est_extrem}
    \end{align}
Hence, $m_p(r, \log |P_s|)=O\Bl \log^2 (1-r)\Br$ as $r\uparrow 1$.

We now prove ii).
Consider the function $$
K(z, \zeta)=\frac 1{1-|\zeta|} \log \Bm \frac {1-z\bar \zeta}{z-\zeta}\Bm, \quad z\in \D, \zeta \in \overline{\D}.
$$
This function has many nice properties. It is nonnegative. Moreover, ($\zeta=\rho e^{i\theta}$, $z=re^{i\vfi}$)
\begin{equation}\label{e:b_prop_K}
K(z, \zeta)=
\frac 1{2(1-\rho)} \log \Bl 1+\frac {(1-r^2)(1-\rho^2)}{r^2-2r\rho \cos(\varphi-\theta)+ \rho^2} \Br,
\end{equation}
and  therefore
$$ \lim_{\rho\to 1-} K(z, \rho e^{i\theta}) =\frac {1-|z|^2}{|\rho e^{i\theta}-z|^2}.$$
We need the following property
\begin{equation}
\label{e:e_prop_K}|K(z, \zeta)|\ge \frac 1{12} \frac{1-|z|^2}{|z-\zeta|^2}, \quad 1-|\zeta|\le \frac12 (1-|z|). \end{equation}
Indeed, since $\log(1+x)\ge \frac x{1+x} $, $x\in (0,1)$, using \eqref{e:b_prop_K}, we deduce that
\begin{equation}\label{e:two_star}
 |K(z, \zeta)| \ge\frac{1+\rho }{2} \frac{1-r^2}{|z-\zeta|^2} \frac{1}{1+\frac{(1-r^2)(1-\rho^2)}{|z-\zeta|^2}}.
\end{equation}
The condition $1-|\zeta|\le \frac12 (1-|z|) $ yields that $|\zeta|\ge \frac12$, and
$$ |z-\zeta|\ge 1-|z|-(1-|\zeta|)\ge \frac{1-|z|}2.$$
Therefore \begin{gather*}
           \frac{1+\rho }{2} \frac{1-r^2}{|z-\zeta|^2} \frac{1}{1+\frac{(1-r^2)(1-\rho^2)}{|z-\zeta|^2}} \ge \frac32
\frac 1{1+\frac{1-\rho^2}{1-r}4(1+r)} \ge\\ \ge \frac 34 \frac{1}{1+2(1+\rho)(1+r)}\ge 
\frac 1{12}.
          \end{gather*}
Inequality \eqref{e:e_prop_K} is proved.

Then we can write that
$ \log|B(z)|= -\sum_{z_k\in Z} K(z, z_k) (1-|z_k|)+ \sum_k \log|z_k| $.

Using \eqref{e:e_prop_K}, we obtain
$$ |\log|B(\rr\theta)||\ge \sum_{z_k\in S(\vfi, \frac{1-r}{2})} K(\rr\vfi, \zeta) (1-|z_k|)\ge
 \sum_{z_k\in S(\vfi, \frac{1-r}{2})} \frac{(1-r^2)(1-|z_k|)}{12|\rr\vfi -z_k|^2}.$$
Elementary geometric arguments  show that $|\rr\vfi -\rho e^{i\theta}| \le |\rr\vfi -e^{i\theta}|$ for $1>\rho\ge r\ge 0$.
It then follows that
\begin{gather*}
 |\log|B(\rr\theta)||\ge\frac 1{12} \sum_{z_k\in S(\vfi, \frac{1-r}{2})} \frac{1-r^2}{|\rr\vfi -e^{i\theta}|^2}(1-|z_k|) \ge \\ \ge \frac 1{3(\frac{\pi^2}{4}+1)} \frac{1-r^2}{(1-r)^2} \sum_{z_k\in S(\vfi, \frac{1-r}{2})} (1-|z_k|)\ge \frac{\sum_{z_k\in S(\vfi, \frac{1-r}{2})} (1-|z_k|)}{3(\frac{\pi^2}{4}+1)(1-r)}.
\end{gather*}
 Recall that $\lambda(\vfi, r)=\sum_{z_k\in S(\vfi, \frac{1-r}2)} (1-|z_k|)$.  From the definition of $S(\vfi,\delta)$ it follows that for fixed $r$ the function $\lambda(\vfi, r)$ is piecewise constant and continuous from the right.  Therefore it attains its maximum on some interval $[\vfi_1(r), \vfi_2(r))$, $\vfi_2(r)>\vfi_1(r)$.
By the assumption of the theorem we deduce that
$$ \frac{C_1}{(1-r)^{1-\gamma}} \ge \biggl(\inp |\log|B(\rr\vfi)|| ^p \, d\vfi\biggr)^{\frac 1p}\ge C_{15} \frac{\biggl(\inp \Bl \lambda(\vfi, r)\Bigr)^p\, d\vfi\biggr)^{\frac 1p} }{1-r}.$$
Hence
$$ \max _\vfi \lambda(\vfi,r) (\vfi_2(r)-\vfi_1(r))^{\frac 1p}\le \Bl\inp \lambda^p (\vfi,r)\, d\vfi\Br^{\frac 1p} \le \frac {C_{1}}{C_{15}} (1-r)^{\gamma }.$$
Tending $p \to \infty$ we obtain the assertion of ii).

This paper was inspired by the ``Conference on Blaschke Products and their Applications'' (Fields Institute, Toronto, July 25–29, 2011) organized by   Javad Mashreghi and Emmanuel Fricain.  I would like to thank the organizers and the staff of the Fields Institute for hospitality and financial support.

\bibliographystyle{amsplain}

\begin{thebibliography}{10}

\bibitem{ChUMZ}  {I. E. Chyzhykov}, \textit{On a complete description of the class of functions without zeros analytic in a disk and having given orders},  {\em	 Ukrainian Math. J.} \textbf{59} (2007), no.7,  1088--1109.



\bibitem{Chy_Grish} {I. E. Chyzhykov}, \textit{Growth of analytic functions in the unit disc and complete measure in the sense of Grishin}, {\em Mat. Stud. }\textbf{{29}} (2008), no.1, 35--44.

\bibitem{Chy_illin} Chyzhykov I. \emph{Argument of bounded analytic functions and Frostman's type conditions,}  Ill. J. Math. {\bf 53} (2009), no.2, 515-531.

\bibitem{CHR1} {I.\ Chyzhykov, J.~Heittokangas and J.~R\"atty\"a}, \textit{{Sharp  logarithmic derivative estimates  with applications to differential  equations in the unit disc}}, {\em  J. Australian Math. Soc.} \textbf{88} (2010), no 2,  145--167.


\bibitem{CH_pams_11} I.\ Chyzhykov, Zero distribution and factorization of analytic functions of slow growth in the unit disc, Proc.Amer. Math. Soc. preprint (accepted)

\bibitem{Dj45} M.M. Djrbashian, On  canonical representation of  functions meromorphic in the unit disk, Dokl. Akad. Nauk Arm. SSR {\bf 3} (1945), no.1, 3--9.

\bibitem{Dj48} M.M. Djrbashian, On the representation problem of analytic functions, Soobshch. Inst. Matem. Mekh. Akad. Nauk Arm. SSR, (1948), no.2, 3--50.


\bibitem{Dj} {M. M. Djrbashian}, \textit{{Integral transforms and
representations of functions in the complex domain}}, Moscow, Nauka,
1966 (in Russian).

\bibitem{Dj73} {M. M. Djrbashian}, Theory of factorization and boundary properties of functions meromorphic in the disc,  Proceedings of the ICM, Vancouver, BC, 1974.



\bibitem{FG} M.\ A. Fedorov, A.\ F.\ Grishin.  \emph{Some questions of the Nevanlinna theory for
the complex half-plane,}  Math. Physics, Analysis and  Geometry
(Kluwer Acad. Publish.) {\bf 1} (1998), no.3, 223--271.

\bibitem{Gr1} A.Grishin, Continuity and asymptotic continuity of
subharmonic  functions, Math. Physics, Analysis, Geometry, ILPTE
{\bf 1} (1994), no.2, 193-215 (in Russian).

\bibitem{Ha} {W.K.Hayman,} Meromorphic functions, Oxford, Clarendon
press, 1964.





\bibitem{Khab06} B.\ Khabibullin, Zero subsets, representation of meromorphic functions and Nevanlinna characteristics in a disc, Mat. Sbornik {\bf 197} (2006), no.2, 117-130.



\bibitem{L_rep} {C.N. Linden,} \textit{The representation of regular functions}, {\em J. London Math. Soc. } \textbf{{39}} (1964), 19--30.

\bibitem{Lin_Bor} {C.N. Linden}, \textit{On a conjecture of Valiron
concerning sets of indirect Borel point}, {\em J.~London Math. Soc. } \textbf{{41}} (1966), 304--312.


\bibitem{Li_mp} {C.N.Linden}, \textit{Integral logarithmic means for regular functions}, {\em Pacific J. of Math.} \textbf{{138}} (1989), no.1, 119--127.

\bibitem{Li92} C.N.Linden, \textit{The characterization of orders for regular functions},
{\em  Math. Proc. Cambrodge Phil. Soc.} \textbf{{111}} (1992), no.2, 299--307.



\bibitem{MaMo} V.I. Matsaev, E.Z. Mogulski, A division theorem for analytic functions with the given majorant and some of its applications, Zapiski Nauchnykh Seminarov POMI,  56 (1976), 73--89 (in Russian).

\bibitem{Na}
{A. G. Naftalevich}, \textit{On interpolating of functions meromorphic in the unit disc},
{\em Dokl Akad. Nauk. SSSR} \textbf{{88}} (1953), no.2, 205--208 (in Russian).



\bibitem{Tsuji} {M. Tsuji}, {\it Potential theory in modern function theory}, Chelsea Publishing Co.
                 Reprinting of the 1959 edition. New York, 1975.

\bibitem{Ts} {M.\ Tsuji}, \textit{Canonical product for a meromorphic function in a unit
circle}, {\em J.\ Math.\ Soc.\ Japan} \textbf{{8}} (1956), no.1, 7--21.


\bibitem{Sha78}
F.A.Shamoyan.  A factorization theorem of M.M.Dzhrbashian's and characteric  of the zeros of analytic  functions in the disk  with a majorant of finite  growth,
Izv. Akad. Nauk. Armyan.SSR Mat.  V.{XIII}  (1978), no. 5--6. -- P.405--422 (in Russian).

\bibitem{Sha83}
F.\ A.\  Shamoyan. Zeros of functions analytic in a disk, that increase near the boundary,
 Izv. Akad. Nauk ArmSSR Ser. Mat,  V.XVIII (1983), no.1, 15--27 (in Russian).

\bibitem{Sha} Shamoyan F.A. Several remarks on
parametric representation of Nevanlinna-Dzhrbashian's classes, Mathematical Notes
{\bf 52} (1992), no 1, 7227--7237.


%
%

\end{thebibliography}

Faculty of Mechanics and Mathematics,

Ivan Franko National University of Lviv,

Universytets'ka 1,  79000, Lviv,  Ukraine

chyzhykov@yahoo.com

Subjclass 2010: {Primary 30J99; Secondary  30D35, 30H15, 37A45}
\end{document}